\documentclass[12pt]{amsart}
\usepackage[letterpaper,margin=1.1in]{geometry}

\usepackage{graphicx,color,amsmath,amssymb,mathrsfs}
\usepackage{amsfonts,amsthm,epsfig,epstopdf, esint,soul, mathtools}
\usepackage[mathscr]{euscript}

\newtheorem{theorem}{Theorem}[section]
\newtheorem{lemma}[theorem]{Lemma}
\newtheorem{proposition}[theorem]{Proposition}
\newtheorem{corollary}[theorem]{Corollary}
\newtheorem{remark}[theorem]{Remark}
\newtheorem{definition}[theorem]{Definition}

\usepackage[utf8]{inputenc}
\usepackage[english]{babel}

\newcommand{\R}{{\mathbb{R}}}
\newcommand{\e}{{\epsilon}}
\newcommand{\g}{{\gamma}}

\newcommand{\M}{{\mathcal{M}}}

\newcommand{\vphi}{\varphi}

\newcommand{\na}{{\nabla}}
\newcommand{\B}{\mathcal{B}}

\newcommand{\vol}{{\rm{vol}}}

\allowdisplaybreaks[4]
\title[Quantitative Stability]{Quantitative Stability for Minimizing Yamabe Metrics with minimal boundary}
\begin{document}
\author{Runze Lin}
\address{Institut de Mathématiques de Toulouse, Université Paul Sabatier, 118, route de Narbonne, 31062 Toulouse
Cedex, France}
\email{runze.lin@math.univ-toulouse.fr}
\author{Bao Yu }
\address{School of Mathematical Sciences, University of Science and Technology of China, Hefei, Anhui Province, P. R. China, 230006}
\email{baoyu1@mail.ustc.edu.cn}

\begin{abstract}
	In this paper, we investigate the stability of minimizing Yamabe metrics on compact manifolds with boundary, in the sense introduced by Escobar. We show that if a function nearly minimizes the Yamabe energy, then the associated conformal metric is quantitatively close to a minimizing Yamabe metric within its conformal class. Moreover, this closeness is controlled by an appropriate power of the Yamabe energy deficit.
\end{abstract}
\maketitle
\section{Introduction}
    Let $(M^n,g)$ be a closed manifold  of dimension $n\ge 3$. The Yamabe problem is to find a metric  $\tilde{g}$, conformal to $g$, such that the scalar curvature of $\tilde{g}$ is a constant. Given a metric $\tilde{g}$ conformal to $g$, i.e.  $\tilde{g} = u^{4/(n-2)} g$ for  a smooth positive function $u$ on $M$, the scalar curvature $R_{\tilde{g}}$ of $\tilde{g}$ is given in terms of $u$ and the scalar curvature $R_g$ of $g$  by
\begin{equation}\label{eqn: comformal change}
R_{\tilde{g}} = u^{1-2^*}\left(-\frac{4(n-1)}{n-2} \Delta u + R_g u \right),
\end{equation}
where $2^* = 2n/(n-2)$. Hence, this problem is to find  a positive solution to the following equation
\begin{equation}\label{eqn: yamabe equation}
    -\Delta u + c_nR_g u = C u^\frac{n+2}{n-2},
\end{equation}
where $c_n =\frac{(n-2)}{4(n-1)}$. In particular, it can be viewed as a variational problem. Solutions to the Yamabe problem are critical points of the following Yamabe energy functional
\begin{equation}\label{eqn: Yamabe functional}
Q(u) =
\ \frac{ \int_M  |\na u|^2 + c_n R_g u^2 \, d\vol_g}
  {\left(\int_M u^{2^*}\, d\vol_g\right)^{2/2^*}} =\frac{c_n\int_M R_{\tilde{g}}\, d\vol_{\tilde{g}} } 
{ \vol_{\tilde{g}}(M )^{2/2^*}}
 \end{equation}
and the Yamabe constant of $(M^n,g)$ is defined as
\begin{equation}
    Y(M,[g])=\inf\{Q(u):u\in W^{1,2}(M), u\ge0\}.
\end{equation}
It is well known that $Y(M,[g])$ is invariant under a conformal change of the metric $g$.
\\

   The Yamabe problem was solved by the combined efforts of the works of Yamabe\cite{Ya60},
   Trudinger\cite{Tr68}, Aubin\cite{Au76a} and Schoen\cite{Sc84}. Yamabe's proof is based on the variational approach. In \cite{Ya60}, he attempted to prove that $Y(M,[g])$ is always achieved. However, Trudinger\cite{Tr68} pointed out that Yamabe's proof is correct only when the Yamabe constant $Y(M,[g]) \leq 0$, and fixed it when $Y(M,[g]) < \epsilon$ for some constant $\epsilon$ (depending on the Sobolev constant). In \cite{Au76a}, Aubin proved the existence under the condition 
   \begin{equation}\label{eq1}
       Y(M,[g]) < Y(S^n,[g_0]),
   \end{equation}
   where $(S^n,g_0)$ denotes the standard metric of the $n$-sphere. Aubin also verified \eqref{eq1} for the case $n \ge 6$ and $(M,g)$ is not locally conformally flat. Schoen\cite{Sc84} established \eqref{eq1} for the remaining cases by the positive mass problem (see\cite{SY79}, \cite{SY88}). A complete proof of the affirmative answer to the Yamabe problem was eventually obtained.

   Once the existence of minimizers is solved, it is natural to ask the stability: if $Q(u)$ is close to the infimum $Y(M,[g])$, to what extent will it affect the distance from $u$ to the minimizers?

   In \cite{BL85}, Brezis and Lieb raised this question on the sphere, asking that given a function $u \in W^{1,2}(S^n)$, whether the energy deficit $Q(u)- Y(S^n,[g_0])$ controls the distance from $u$ to the family of minimizers. By means of explicit characterization of the minimizers on the round sphere by Aubin\cite{Au76b} and Talenti\cite{Ta76}(see also Obata\cite{Ob71}), Bianchi and Egnell give a positive answer: for any nonnegative $u \in W^{1,2}(S^n)$
   \begin{equation}\label{eqn: BE}
   Q_{(S^n,g_0)}(u) -Y(S^n, [g_0]) \geq c  \left(\frac{\inf \left\{ \| u-v\|_{W^{1,2}(S^n)}\ | \  {v \in \mathcal{M}_{(S^n,g_0)}} \right\}}{\| u \|_{W^{1,2}(S^n )} }\right)^2,
   \end{equation}
   where $\mathcal{M}$ denotes all minimizers of the functional $Q$, $c$ only depends on the dimension, and the exponent $2$ is sharp under the 
   $W^{1,2}$ norm. However, for a general closed Riemannian manifold, the minimizers are not known in any explicit form. By elegantly integrating the Lyapunov-Schmidt reduction with the \L ojasiewicz distance inequality, Engelstein, Neumayer, and Spolaor\cite{ENL22} derived the following quantitative stability result.
    \begin{theorem}\cite{ENL22} \label{ENL22}
 	Let $(M^n,g)$ be a $C^\infty$ closed Riemannian manifold of dimension $n\geq 3$ that is not conformally equivalent to the round sphere. There exist constants $c>0$ and  $\g\ge 0 $, depending on $(M,g)$, such that 
 	\begin{equation}
 	Q(u) - Y(M,[g]) \geq c\, d(u, \mathcal{M})^{2+\g}\, \qquad \forall u \in W^{1,2}(M;\R_+)  \,.
 	\end{equation}
 	Moreover, there exists an open dense subset $\mathcal G$ in the $C^2$ topology on the space of $C^\infty$-conformal classes of metrics on $M$ such that if $[g] \in \mathcal G$, we may take $\g=0$.
   \end{theorem}
   Furthermore, their result is optimal, as the existence of manifolds with a strictly positive exponent $\gamma$ is proven.
   
   Motivated by the question mentioned above, our purpose in this paper is to establish a stability result for a compact Riemannian manifold with boundary.
   From now on, $(M^n,g)$ denotes some smooth compact n-dimensional Riemannian manifold with boundary unless we specify otherwise. We use $L_g$ to denote $- \Delta + c_nR_g$, $B_g$
   to denote $\frac{\partial}{\partial \nu} +\frac{n-2}{2}h_g$, $\nu$ to denote the outward unit normal vector on $\partial M$ with respect to $g$, and $h_g$ to denote the mean curvature of $\partial M$. 
   So, the Yamabe problem corresponding to $M$ is to find a metric  $\tilde{g}$, conformal to $g$, such that the scalar curvature of $\tilde{g}$  and the mean curvature for $\partial M$ of $\tilde{g}$  are both constants.
   \\

   Let $u>0$ be some positive function on $M$, let $\tilde{g}= u^\frac{4}{n-2}g$, and calculate the mean curvature $h_{\tilde{g}}$ as 
   \begin{equation}
       h_{\tilde{g}}=\frac{2}{n-2}u^{-\frac{n}{n-2}}B_g u.
   \end{equation}
   There are two types of Yamabe problem with boundary:\\
   Type (i)\cite{Es92a}: find a metric $\tilde{g}$ conformal to $g$ such that the scalar curvature $R_{\tilde{g}}$ is constant in $M$ and the mean curvature $h_{\tilde{g}}$ vanishes on $\partial M$.\\
   Type (ii)\cite{Es92b}: find a metric $\tilde{g}$ conformal to $g$ such that the scalar curvature $R_{\tilde{g}}$ vanishes in $M$ and the mean curvature $h_{\tilde{g}}$ is constant on $\partial M$.

   Combined with \eqref{eqn: comformal change}, we can get the following two Yamabe equations corresponding to Type (i) and Type (ii)
   \begin{equation}\label{eqn:boundary}
       \begin{cases}
           L_g u= C u^\frac{n+2}{n-2}  ~~ \text{in}  ~~M^\circ  ,\\
       B_g u= 0   ~ ~\text{on}  ~\partial M,\\
       ~~~~u\geq 0  ~~\text{in}~ M. 
       \end{cases}
    \end{equation}
    \begin{equation}
       \begin{cases}
           L_g u=0 ~~ \text{in}  ~~M^\circ  ,\\
       B_g u= C u^\frac{n}{n-2}   ~ \text{on}  ~\partial M,\\
       ~~~u \geq  0 ~\text{in} ~ M.
       \end{cases}
    \end{equation}
    In this paper, we will discuss the Type (i) Yamabe problem. Compared to the case of the closed manifold, there is also an energy functional
    \begin{equation}
        Q(u) =
\ \frac{ \int_M  |\na u|^2 + c_nR_g u^2 \, d\vol_g+ \frac{n-2}{2}\int_{\partial M}h_g u^2 d \sigma_g}
  {\left(\int_M u^{2^*}\, d\vol_g\right)^{2/2^*}}.
    \end{equation}
    Moreover, the Yamabe constant is defined as
    \begin{equation}
        Y(M,\partial M, [g])=\inf\{Q(u):u\in W^{1,2}(M), u\ge0\}.
    \end{equation}
    $Y(M,\partial M, [g])$ is invariant under a conformal change of the metric $g$ and $Y(M,\partial M, [g]) \leq Y(S^n_+,\partial S^n_+,[g_0])$, where $Y(S^n_+,\partial S^n_+,[g_0])$ denotes the Yamabe constant
    of the hemisphere $S^n_+$ equipped with the standard metric.

    In \cite{Es92a}, Escobar showed the existence of the Type(ii) Yamabe problem provided $Y(M,\partial M, [g]) < Y(S^n_+,\partial S^n_+,[g_0])$, and he also proved that $Y(M,\partial M, [g]) < Y(S^n_+,\partial S^n_+,[g_0])$ holds in the dimension $3\leq n \leq 5$ if $M$ is not conformally equivalent to the hemisphere $S^n_+$. In dimension $n\geq 6$, Escobar was able to verify this inequality under the assumption that $\partial M$ is not umbilic. Brendle and Chen\cite{BC14}
    considered the remaining case: $n\geq 6$ and $\partial M$ is umbilic. They verified the remaining case subject to the validity of the Positive Mass Theorem.
    
    In this paper, we consider the stability of minimizers of the Yamabe metrics. For the type(ii) Yamabe problem, Borquez, Caju and Hanne\cite{Ha26} proved the quantitative stability for the Yamabe minimizers using the method of Lyapunov-Schmidt reduction and the \L ojasiewicz distance inequality. So we will handle the type(i) problem.
    
    Let $\M \subset W^{1,2}(M)$ denote the set of all minimizers of $Q(u)$. Define 
    \begin{equation}
        d(u,\M) = \frac{\inf\left\{ \| u-v\|_{W^{1,2}(M)} \ | \ v \in \M\right\}}{\| u\|_{W^{1,2}(M)}}.
    \end{equation}
    Our result is a quantitative stability estimate for minimizers of the Yamabe functional.
    
\begin{theorem}\label{thm: minimizers}
 	Let $(M^n,g)$ be a $C^\infty$ compact Riemannian manifold of dimension $n\geq 3$ with boundary such that $Y(M,\partial M,[g]) < Y(S^n_+,\partial S^n_+,[g_0]) $. There exist constants $c>0$ and  $\g\ge 0 $, depending on $(M,g)$, such that 
 	\begin{equation}\label{e: minimizers}
 	Q(u) - Y(M,\partial M,[g]) \geq c\, d(u, \mathcal{M})^{2+\g}\, \qquad \text{for all } u \in W^{1,2}(M;\R_+)  \,.
 	\end{equation}
 \end{theorem}
 \begin{remark}
     Similarly to the closed case, we expect the quadratic stability in Theorem \ref{ENL22} to hold generically, i.e. there exists an open dense subset $\mathcal G$ in the $C^2$ topology on the space of $C^\infty$-conformal classes of metrics on $M$ such that if $[g] \in \mathcal G$, we may take $\g=0$. This property comes from Schoen \cite{Sc91} that generically, there are finitely many critical points of $Q$ and each one is non-degenerate. However, these tools are still lacking in the boundary setting.
 \end{remark}
 \begin{remark}
     In the closed case, there are some examples (See in \cite{Sc91}, \cite{CCR15} and \cite{Fr22}) with super quadratic growth, i.e. $\gamma>0$. In \cite{Fr22}, Frank showed that on the manifold $S^1(1/\sqrt{d-2})\times S^{d-1}(1)$ with the standard product metric, the stability exponent is $4$. One can simply take $S^1(1/\sqrt{d-2})\times S^{d-1}_+(1)$ and apply the above result.
     So we can get the example with $\gamma >0$ in the minimal boundary case.
     
 \end{remark}
We are also able to obtain  a corollary from Theorem \ref{thm: minimizers}, a conformally invariant stability estimate. We can define the following distance
\begin{equation}
    ||g_u - g_v|| =\left(\int_M |u-v|^{2^*} d \vol_g   \right) ^\frac{1}{2^*},
\end{equation}
where $g_u = u^\frac{4}{n-2}g$, and $2^*=\frac{2n}{n-2}$. Similarly, in the case of $Y(M,\partial M,[g])\ge 0$, we define
\begin{equation}
    ||g_u - g_v||_*= \left( \int_M \frac{1}{c_n} |\na u- \na v|^2 + R_g (u-v)^2 d \vol_g +2(n-1)\int _{\partial M}h_g (u-v)^2  \right)^\frac{1}{2},
\end{equation}
for any $g\in \M(M,g)$ with $\vol_g(M)=1$. We may claim that although $||-||$ and $||-||_*$ are both defined with respect to a fixed conformal representative $g\in[g]$, but it is independent of this choice.(see, e.g. in \cite{ENL22})

 \begin{corollary}
 Let $(M^n,g)$ be a compact $C^\infty$ Riemannian manifold of dimension $n\geq 3$ with boundary such that $Y(M,\partial M,[g]) < Y(S^n_+,\partial S^n_+,[g_0]) $. There exist constants $c>0$ and  $\g\ge 0 $, depending on $M$ and $[g]$, such that 
 	\begin{equation}\label{e: minimizers conf invar}
 	\mathcal{E}_g  - Y(M,\partial M,[g]) \geq c\, \left(\frac{\inf\{ \| g - \tilde g\| : \tilde{g} \in \M \}}{\vol_g(M)^{1/2^*}}\right)^{2+\g}\, \qquad \forall g \in [g]\,.
 	\end{equation}
 	Here $\mathcal{E}_g= c_n \vol_g(M)^{-2/2^*}\,\left(\int_M R_{g} d\vol_g +2(n-1) \int_{\partial M} h_g d\sigma_g\right)  $ .\\
 When $Y= Y(M,\partial M,[g]) \geq 0$ and $\mathcal{E}_g  - Y(M,\partial M,[g])  \leq 1$, there exist constants $c>0$ and  $\g\ge 0 $ depending on $M$ and $[g]$ such that 
	\begin{equation}\label{e: minimizers conf invar2}
 	\mathcal{E}_g  - Y(M,\partial M,[g]) \geq c\, \left(\frac{\inf\{ \| g - \tilde g\|_* : \tilde{g} \in \M \}}{\vol_g(M)^{ 1/2^*}}\right)^{2+\g}\, \qquad \forall g \in [g]\,.
 	\end{equation}
 \end{corollary}
 This corollary can be directly derived from Theorem \ref{thm: minimizers} and the Sobolev inequality. See \cite{ENL22} for more details on the conformal invariance of the involved quantities.

\section{Properties of the Yamabe Energy and Lyapunov-Schmidt reduction}
    Recall the Yamabe energy:
    \begin{equation}
        Q(u) =
\ \frac{ \int_M  |\na u|^2 + c_nR_g u^2 \, d\vol_g+ \frac{n-2}{2}\int_{\partial M}h_g u^2 d \sigma_g}
  {\left(\int_M u^{2^*}\, d\vol_g\right)^{2/2^*}}.
    \end{equation}
Since $Q(cu ) = Q(u)$ for any $c >0$,  it will often be easier to work with functions that have $L^{2^*}$ norm equal to  $1$. To that end we introduce the following Banach manifold:

\begin{equation}\label{e:defofB}
 \B = \left\{u\in W^{1,2}(M ; \R_+) \mid \int_M u^{2^*}\ d\vol_g = 1\right\}.\end{equation}

 Note that the collection of metrics represented by \eqref{e:defofB} is conformally invariant, this can be seen in the equivalent condition that the metric $g_u = u^{4/(n-2)}g$ has unit volume. 

 It will be useful to give some additional definitions. Denote $\mathcal{M} \subset W^{1,2}(M)$ the set of all miminzers of $Q(u)$, and let  $\mathcal{CSC} \subset W^{1,2}(M)$ be the set of all critical points of $Q(u)$,
 i.e. solutions to \eqref{eqn:boundary} for some $C \in \R$. 
 Furthermore, let $\M_1 := \M \cap \B$ and $\mathcal{CSC}_1:=\mathcal{CSC}\cap \B$, which are normalized minimizers and critical points. Finally, for any $v \in \B$, we denote by $\B(v,\delta)$ the $W^{1,2}$ ball of radius $\delta$ centered at $v$ inside of $\B$.

 Since we will do the variation of $Q$ in the space $\B$, we define the tangent space for any $v\in \B$
 \begin{equation}
     T_v \B = \left\{u \in W^{1,2}(M) \mid  \int_{M} v^{2^*-1} u \, d \vol_g =0\right\},
 \end{equation}
and the associated orthogonal projection:
\begin{align}
    \pi_{T_v\B} : ~& W^{1,2}(M) \to T_v\B\subset W^{1,2}(M)\\
    &~~~~~u~~~~~\mapsto  u - \left(\int v^{2^*-1} u\right)v .
\end{align}
We start by computing the first and second variations of $Q$ on $\B$, Throughout the paper, we will write $\na_\B Q(u) = \na Q(u) \circ \pi_{T_u \B}$ and the same for $\na ^2 _\B Q(u)$.
    
\begin{lemma}\label{lem:bm}
    For every $u \in \B$, the first and second variations of $Q$ on $\B$ are given by
    \begin{equation}\label{e:firstvariation}
    \nabla_{\B} Q(u)[\varphi] = 2\int_{M} \left(-\Delta u + c_nR_gu\right)\pi_{T_u\B}\varphi\, d\vol_g\, +2\int_{\partial M} \left( \frac{\partial u}{\partial \nu}+\frac{n-2}{2}h_g u   \right)\pi_{T_u\B}\varphi d\sigma_g ,
    \end{equation}
	\begin{equation}\label{e:secondvariation1}
	\begin{aligned}
	&\frac{1}{2}\nabla_\B^2 Q(u)[\varphi,\eta]
	 = \int_M \left\{  \,\na \pi_{T_u\B}\varphi\,\cdot \na \pi_{T_u\B}\eta + c_n R_g\,(\pi_{T_u\B}\varphi) \,(\pi_{T_u\B}\eta) \right\}\, d\vol_g\\ &+\frac{n-2}{2}\int_{\partial M}h_g(\pi_{T_u\B}\phi) (\pi_{T_u\B}\psi) d\sigma_g-(2^*-1)Q(u)\int_{M} u^{2^*-2}\, (\pi_{T_u\B}\varphi) \,(\pi_{T_u\B}\eta)\, d\vol_g,\end{aligned}
	\end{equation}
 for all $ \varphi,\eta\in W^{1,2}(M)$.	We will often omit the projection maps when we are doing computations with $\nabla_{\B}^2Q$.

 Moreover, the following two continuous properties of the Hessian operator hold.

 Let $\mathcal{L}_v \phi = L_g \phi -(2^*-1)Q(v)v^{2^*-2} \phi $, for $u,w \in C^{2,\alpha} \cap \B$ sufficiently close, there exists a modulus of continuity $\omega$, with $\omega(r) \to 0$ as $r\to 0$, such that  for all $\eta \in C^{2,\alpha}(M)$,
 \begin{equation}\label{prop1}
     ||(\mathcal{L}_u \pi_{T_u{\B}} \eta, B_u \pi_{T_u{\B}} \eta )-(\mathcal{L}_w \pi_{T_w{\B}} \eta, B_w \pi_{T_w{\B}} \eta )||_{C^{0,\alpha}(M) \times C^{1,\alpha}(\partial M)} \leq \omega(||u-w||_{C^{2,\alpha}(M)})||\eta||_{C^{2,\alpha}(M)}.
 \end{equation}

 For $u,w \in \B$ sufficiently close in $W^{1,2}$, there exists a modulus of continuity $\omega$, with $\omega(r) \to 0$ as $r\to 0$, such that  for all $\vphi, \eta \in W^{1,2}(M)$,
 \begin{equation}\label{prop2}
     |\na ^2 Q_\B(u)[\vphi,\eta]-\na ^2 Q_\B(w)[\vphi,\eta]|\leq w(||u-w||_{W^{1,2}(M)})||\vphi||_{W^{1,2}(M)}||\eta||_{W^{1,2}(M)}.
 \end{equation}

\end{lemma}
\begin{proof} 
	  Let us denote
	$$
	\mathcal E(u):=\int_M  |\nabla u|^2+ c_n R_gu^2\,d\vol_g\,+\frac{n-2}{2}\int_{\partial M} h_g u^2 d\sigma_g,
	$$ 
	and observe that if $u\in \B$, then $\mathcal E(u)=Q(u)$. For $u\in \B$ and $\varphi \in W^{1,2}(M)$,  we can compute the first variation of $Q$ at points of $\B$:
	\begin{align}
	\nabla Q(u)[\phi] &:= \frac{d}{dt}(Q (u + t\phi))\Big|_{t=0}\\ 
	&=\bigg( [\vol(M, g_{u+t\phi})]^{-2/2^*}\, 2\big[\int_M\left(( \nabla u\cdot\nabla \phi+\frac{t}{2}\,|\nabla \phi|^2)+c_n R_g(u\,\phi+\frac{t}{2}\phi^2)\right)\,d\vol_g \notag\\
    &+ \frac{n-2}{2}\int_{\partial M}h_g(u\phi+\frac{t}{2}\phi^2)d\sigma_g \big] \notag \\
	 &  -2\,\left[\vol(M, g_{u+t\phi}) \right]^{-(2+2^*)/2^*}\, \left[\int_M (u+t\phi)^{2^*-1} \,\phi \,d\vol_g\right]  \,\mathcal E(u+t\phi)\bigg)\Big|_{t=0} \notag 
	\\
	&= 2 \int_{M} \left(-\Delta u + c_nR_gu-Q(u)u^{2^*-1}\right)\phi\, d\vol_g\, + 2\int_{\partial M} \left( \frac{\partial u}{\partial \nu}+\frac{n-2}{2}h_g u   \right)\phi d\sigma_g.
    \end{align} 
	In particular, when restricted to the tangent space of $\B$, we have
	$$
	\nabla_{\B} Q(u)[\varphi] = 2\int_{M} \left(-\Delta u + c_nR_gu\right)\pi_{T_u\B}\varphi\, d\vol_g\, +2\int_{\partial M} \left( \frac{\partial u}{\partial \nu}+\frac{n-2}{2}h_g u   \right)\pi_{T_u\B}\varphi d\sigma_g .
	$$

	Differentiating \eqref{e:firstvariation}, we obtain
   \begin{align}
       &\nabla^2Q(u)[\phi,\psi]\\
       =& [\vol(M, g_{u})]^{-2/2^*}2\left[\int_M \left(\na \phi \cdot \na \psi +c_n R_g\phi \psi  \right)d\vol_g +\frac{n-2}{2}\int_{\partial M}h_g\phi \psi d\sigma_g    \right] \notag \\
       &-2[\vol(M, g_{u})]^{-2/2^*-1}\left[\int_M u^{2^*-1}\psi d\vol_g  \right]2\cdot\left(\int_M(\na u \cdot \na \phi +c_nR_g u \phi +\frac{n-2}{2}\int_{\partial M}h_g u \phi d\sigma_g \right) \notag \\
       &+2\cdot(2+2^*)[\vol(M, g_{u})]^{-2/2^*-2}\int_M u^{2^*-1}\psi d\vol_g\int_M u^{2^*-1}\phi d\vol_g \mathcal E(u)\notag \\
       &-2[\vol(M, g_{u})]^{-2/2^*-1}(2^*-1)\int_M u^{2^*-2} \phi \psi d\vol_g \mathcal E(u) \notag \\
       &-2[\vol(M, g_{u})]^{-2/2^*-1}\int_M u^{2^*-1} \phi d\vol_g \frac{d}{dt}\large|_{t=0} \mathcal E(u+t\psi) \notag \\
       =&2\left[\int_M \left(\na \phi \cdot \na \psi +c_n R_g\phi \psi  \right)d\vol_g +\frac{n-2}{2}\int_{\partial M}h_g\phi \psi d\sigma_g    \right] -2(2^*-1)Q(u)\int_Mu^{2^*-2}\phi \psi d\vol_g \notag \\
       &-2\int_M u^{2^*-1} \phi d\vol_g \frac{d}{dt}\large|_{t=0} \mathcal E(u+t\psi)-2\int_M u^{2^*-1} \psi d\vol_g \frac{d}{dt}\large|_{t=0} \mathcal E(u+t\phi)\notag \\
       &+2\cdot (2+2^*) Q(u) \int_M u^{2^*-1}\psi d\vol_g\int_M u^{2^*-1}\phi d\vol_g. \notag 
   \end{align}
   In particular, when restricted to the tangent space of $\B$, we have
   \begin{align}
       &\nabla_\B^2Q(u)[\phi,\psi]=2\bigg[\int_M \left(\na \pi_{T_u\B}\phi \cdot \na \pi_{T_u\B}\psi +c_n R_g\pi_{T_u\B}\phi \pi_{T_u\B}\psi  \right)d\vol_g\\
       &+\frac{n-2}{2}\int_{\partial M}h_g\pi_{T_u\B}\phi \pi_{T_u\B}\psi d\sigma_g\bigg]
       -2(2^*-1)Q(u)\int_Mu^{2^*-2}\pi_{T_u\B}\phi \pi_{T_u\B}\psi d\vol_g. \notag
   \end{align}
   To conclude the proof, let us check the continuity. From the definition of $\pi_{T_u{\B}}$, let $a_u(\eta)= \int_M u^{2^* -1} \eta d \vol_g$, we have 
   \begin{align*}
       \pi_{T_u{\B}} \eta -\pi_{T_w{\B}} \eta &= -a_u(\eta)u +a_w(\eta) w\\
       &=-a_u(\eta) (u-w) -(a_u(\eta) -a_w(\eta))w.
   \end{align*}
   It is easy to show $||\pi_{T_u{\B}} \eta -\pi_{T_w{\B}} \eta||_{C^{2,\alpha}(M)} \leq \omega(||u-w||_{C^{2,\alpha}(M)})||\eta||_{C^{2,\alpha}(M)}.$
   Now compute 
   \begin{align} \label{Prop1}
       &\mathcal{L}_u \pi_{T_u{\B}} \eta -\mathcal{L}_w \pi_{T_w{\B}} \eta 
       =L_g (\pi_{T_u{\B}} \eta -\pi_{T_w{\B}} \eta) -(2^*-1)[Q(u)u^{2^*-2}\pi_{T_u{\B}} \eta -Q(w)w^{2^*-2}\pi_{T_w{\B}} \eta] \notag\\
       =&L_g (\pi_{T_u{\B}} \eta -\pi_{T_w{\B}} \eta) -(2^*-1)(Q(u)u^{2^*-2}-Q(w)w^{2^*-2})\pi_{T_u{\B}} \eta - (2^*-1)Q(w)w^{2^*-2}(\pi_{T_u{\B}} \eta -\pi_{T_w{\B}} \eta).
   \end{align}
   Combined with the continuity of $L_g: C^{2,\alpha}(M) \to C^{0,\alpha}(M)$ and $B_g : C^{2,\alpha}(M) \to C^{1,\alpha}(\partial M) $, we can verify \eqref{prop1}.

   For the $W^{1,2} $ continuity, by the Sobolev embedding $W^{1,2} (M) \hookrightarrow L^{2^*}(M)$,
   we get 
   \begin{equation*}
       |a_u(\eta)| \leq ||u||_{L^{2^*}}^{2^*-1} ||\eta ||_{L^{2^*}} \leq C ||\eta ||_{W^{1,2}(M)}.
   \end{equation*}
   Moreover, the map $u \mapsto u^{2^*-1}$ is continuous from $L^{2^*}$ to $L^{\frac{2^*}{2^*-1}
   }$ on nonnegative functions, hence
   \begin{equation*}
       |a_u(\eta)-a_w(\eta)|\leq ||u^{2^*-1}-w^{2^*-1}||_{L^{\frac{2^*}{2^*-1}}}||\eta||_{L^{2^*}}
       \leq \omega(||u-w||_{W^{1,2}(M)}) ||\eta ||_{W^{1,2}(M)}.
   \end{equation*}
   It follows that $||\pi_{T_u{\B}} \eta -\pi_{T_w{\B}} \eta||_{W^{1,2}(M)} \leq \omega(||u-w||_{W^{1,2}(M)})||\eta||_{W^{1,2}(M)}.$
   Then, the same argument as \eqref{Prop1}, we can verify \eqref{prop2}.
   
\end{proof}

\subsection{Lyapunov-Schmidt Reduction} 

Given $v \in \mathcal{M}_1$, we let $K= \ker \nabla^2_{\B} Q(v)[-,-]\subset T_v\B$, thinking of $\nabla^2_{\B} Q(v)[-,-]$ as  an operator from $T_v\B\subset W^{1,2}(M) \rightarrow H^{-1}(M)$. Consider the effect of the boundary and the formula \eqref{e:secondvariation1}, the kernel is
\begin{equation}
    K=\{ \phi \in T_v\B|-\Delta \phi +c_n R_g \phi -(2^*-1)Q(v)v^{2^*-2}\phi =0 ~\text{in}~ M^\circ, B_g \phi =0 ~ \text{on}~ \partial M\}
\end{equation}
Since it is generated by an elliptic operator with Robin boundary condition on a compact manifold(See \cite{ADN64} or \cite{GT}), we know $\dim K := l < \infty$. We let $K^\perp$ denote the orthogonal complement of $K$ in $T_v\B$ with respect to the $L^2$ inner product. We parametrize a small neighborhood of $v$ in $\B$ by the tangent space $T_v \B$ with
\begin{align*}
    \Psi: &T_v\B \to B  \\
    &~~~\xi ~\mapsto \frac{v+\xi}{||v+\xi||_{L^{2^*}(M)}}.
\end{align*}
It is easy to see that $\Psi(\xi) \in \B $, $\Psi(0)= v$ and $\na \Psi (0) = Id$. Define $\hat{Q}(\xi)=Q(\Psi(\xi))$.
\begin{lemma}[Lyapunov-Schmidt Reduction]\label{lem: LS reduction}
	Let $(M,g)$ be a compact Riemannian manifold with boundary with $g\in C^3$ and fix $v \in \mathcal{\M}_1$. There is a open neighborhood $U\subset K$ of $0$ in $K$
	and a map 
	\[
	F: U \to K^{\perp} 
	\]
	 with $F(0)=0$ and $\nabla F(0)=0$ satisfying the following properties:
\begin{enumerate}
	\item\label{item: LS property 1} Let $q: U\to \R$ be the function defined by $q(\vphi)= \hat{Q}( \vphi + F(\vphi))$  such that
    \begin{align}
        &\na q(\vphi)[\eta]= \na \hat{Q}(\vphi +F(\vphi))[\eta]  ~~\forall \eta \in K,\\
        &\na \hat{Q}(\vphi +F(\vphi))[\zeta]=0~~ \forall \zeta \in K^\perp.
    \end{align}
	Furthermore, $\vphi \mapsto q(\vphi)$ is real analytic.

\item \label{item: prop2}$\na q(\vphi)=0$ if and only if $\Psi(\vphi+F(\vphi))$  is a normalized critical point of $Q$.

\item\label{item schauder estimates} There exists $C$ such that for all $ \vphi \in  U$ and $\eta \in K$, we have
		\begin{equation}\label{eqn: estimates for DF}
	\begin{split}
\| \nabla F(\vphi)[\eta] \|_{C^{2,\alpha}}  &\leq C \|\eta\|_{C^{0,\alpha}}\,.
	\end{split}
	\end{equation}

\end{enumerate}
\end{lemma} 

If $v$ is no-degenerate, which means that $K=\text{ker} \na^2 Q(v)$ is trivial, then this lemma says nothing. If $K$ is not trivial, this lemme reduces the infinite-dimensional stability problem near a critical point to a finite-dimensional analytic problem, while the complementary directions remain quadratically controlled by the second variation. This reduction has been already performed for Yamabe functional in a variety of contexts(see, e.g. \cite{CCR15} \cite{ENL22} \cite{Ha26}). 

\begin{proof}
    The map $F$ is constructed by solving the $K^\perp$- component of the Euler–Lagrange equation via the implicit function theorem. Since the Hessian is invertible on $K^\perp$, for each kernel parameter $\vphi$ there is a unique correction $F(\vphi) \in K^\perp$ such that the gradient has no $K^\perp$-component at $\Psi(\vphi +F(\vphi))$. Hence the full critical point equation reduces to the $K$-component only, that is, to a finite-dimensional equation for the reduced energy q.

   Let $X=C^{2,\alpha}(M) \cap T_v\B$ with the $C^{2,\alpha}(M)$ norm, $Y=\left( C^{0,\alpha}(M), C^{1,\alpha}(\partial M) \right)$ with the norm $||(f,g)||_Y = ||f||_{C^{0,\alpha}(M)}+||g||_{C^{1,\alpha}(\partial M)}$ and the pairing
   \begin{equation}
       \langle (f,g), \zeta \rangle = \int_M f \zeta d\vol_g + \int_{\partial M} g \zeta d\sigma_g.
   \end{equation}
   Let $K^\perp$ denote the orthogonal complement of $K$ in $T_v\B$ with respect to the $L^2$ inner product, and $W= C^{2,\alpha}(M) \cap K^\perp$. Then $X=K \oplus W$.
   Define $\mathcal{L}_v \phi = L_g \phi -(2^*-1)Q(v)v^{2^*-2} \phi $ and $\mathcal{A} : X \to Y$ with $\mathcal{A}\xi =(\mathcal{L}_v\xi, B_g\xi) $. Then $\langle \mathcal{A} \xi , \zeta\rangle = \frac{1}{2} \na_\B^2 Q(v)[\xi,\zeta]$.

   Let $N_{W} = \{(f,g) \in Y: \langle (f,g), \zeta \rangle =0 ~\text{  for every} ~\zeta \in W \} $. Define the quotient space $Z=Y/N_{W}$ with the minimal norm $||[(f_0,g_0)]||_Z= \min_{(f,g) \in [(f_0,g_0)]}||(f,g)||_Y$. Then the operator $\bar {\mathcal{A}} : W \to Z$ given by $\bar {\mathcal{A}} \zeta =[{\mathcal{A}}\zeta]$ is an isomorphism.

   Indeed, to see injectivity directly, suppose $w\in W$, and $[{\mathcal{A}}w]=0$. Then
   $\frac{1}{2} \na_\B^2 Q(v)[w,\zeta]=\langle \mathcal{A} w , \zeta\rangle =0$ for all $\zeta \in W$. Also, for every $\xi \in K$, $ \na_\B^2 Q(v)[w,\xi]=0$. Therefore $w \in K$. But $w\in W$, so $w=0$. Since $\mathcal{A}$ is an elliptic operator of robin boundary condition,  it is Fredholm of index zero. Together with $Y=\mathcal{A}(W) \oplus N_W$, we see that $\bar{\mathcal{A}}$ is an isomorphism. We also have the following Schauder estimate
   \begin{equation} \label{eq7}
       ||w||_{C^{2,\alpha}(M)} \leq C || \bar {\mathcal{A}} w||_Z \leq C||\mathcal{A} w||_Y,~~  \forall  ~w \in W.
   \end{equation}

   For $\xi \in X$, set $u_\xi = \Psi(\xi)$, and define the residual
   \begin{equation}
       \mathcal{R}(\xi)= ||v+\xi||_{L^{2^*}} ^{-1} (L_g u_\xi -\hat{Q}(\xi)u_\xi ^{2^*-1}, B_gu_\xi) \in Y.
   \end{equation}
   Then $\mathcal{R}(0) = 0$ and  $\na \hat{Q}(\xi)[\zeta] = 2 \langle \mathcal{R}(\xi), \zeta\rangle $ for all $\xi,\zeta \in X$. Furthermore, $\na \mathcal{R}(0)[\zeta] = \mathcal{A} \zeta$ for all $\zeta \in W$.

    Now we construct $F$ by defining $\mathcal{G}: X \to W$ given by $\mathcal{G}(\xi)= \bar{\mathcal{A}}^{-1}(\mathcal{R}(\xi))$ and $\mathcal{N} : X \to X$ given by $\mathcal{N}(\xi)= \pi _K \xi + \mathcal{G}(\xi)$.
    
    Then $\mathcal{N}(0)=0$ and $\na \mathcal{N}(0)[\xi]=\pi _K \xi + \na \mathcal{G}(0)[\xi] =\pi _K \xi + \bar{\mathcal{A}}^{-1}(\mathcal{A}(\xi)) = \pi _K \xi +\pi _{W} \xi$. Hence $\na \mathcal{N}(0) = Id_X$. By the analytic inverse function theorem, after shrinking neighborhoods if necessary, $\mathcal{N}$ has a real analytic inverse $\mathcal{N}^{-1}$ defined in a neighborhood of $0 \in X$. Let $U=K \cap \text{Dom}(\mathcal{N}^{-1})$. For $\vphi \in U$, define 
    \begin{equation}
        F(\vphi)= \pi_{W} \mathcal{N}^{-1} (\vphi).
    \end{equation}
    Then we will introduce some properties of $F$.
    Let us make some initial observations that will be useful in proving the claimed properties of $F$. 
    For any $\vphi \in U$, from the definition of $\mathcal{N}$ we have
    \begin{equation} \label{eq6}
        \vphi = \mathcal{N}(\mathcal{N}^{-1}(\vphi)) = \pi_K \mathcal{N}^{-1}(\vphi) + \mathcal{G}(\mathcal{N}^{-1}(\vphi)).
    \end{equation}
    Taking the $K$-component gives $\pi_K \mathcal{N}^{-1}(\vphi)= \vphi$. Therefore, 
    \begin{equation}
        \mathcal{N}^{-1}(\vphi) =\pi_K \mathcal{N}^{-1}(\vphi) + \pi_{W} \mathcal{N}^{-1} (\vphi) =\vphi +F(\vphi).
    \end{equation}
    Differentiating it, we find that for any $\vphi \in U$ and $\eta \in K$
    \begin{align}
    \pi_K \na \mathcal{N}^{-1}(\vphi)[\eta ] &= \eta, \notag\\
    \pi_{W} \na \mathcal{N}^{-1}(\vphi)[\eta ] &=\na F(\vphi)[\eta]. \label{eq4}
    \end{align}
    Thus $\na F(0)[\eta] = \pi_{W} \na \mathcal{N}^{-1}(0)[\eta ] =\pi_{W} \eta = 0$. Hence  $\na F(0)=0$, and it is clear that $F(0)=0$.

    Now, we define $q(\vphi)=\hat{Q}(\vphi +F(\vphi))$ for $\vphi \in U$. Then $q$ is real analytic because both $\hat{Q}$ and $F$ are real analytic.

    For $\eta \in K$, the chain rule gives
    \begin{equation}
        \na q(\vphi)[\eta] =\na \hat{Q}(\vphi +F(\vphi))[\eta +\na F(\vphi)[\eta]].
    \end{equation}
    Since $\na F(\vphi)[\eta] \in W$ and $\na \hat{Q}(\vphi +F(\vphi)) $ vanishes on $W$ from \eqref{eq6}, we get $\na q(\vphi)[\eta]= \na \hat{Q}(\vphi +F(\vphi))[\eta]$.

    It remains to prove \eqref{eqn: estimates for DF}. Fix $\vphi \in U$ and $\eta \in K$. Set
    \begin{equation}
        \xi_\vphi = \vphi +F(\vphi),~~~z =\na F(\vphi)[\eta] \in W,~~ h=\eta +z.
    \end{equation}
    From \eqref{eq4}, we know $h=\na \mathcal{N}^{-1}(\vphi)[\eta ]$. Then differentiating $\vphi = \mathcal{N}(\mathcal{N}^{-1}(\vphi))$ in the direction $\eta$, we get $\na \mathcal{N}(\xi_\vphi)[h]=\eta$.

    Since $\pi_K h =\eta $ and $\mathcal{N}(\xi)= \pi _K \xi + \mathcal{G}(\xi)$, we obtain $\na \mathcal{G}(\xi_\vphi)[h]=0$, and $\na \mathcal{R}(\xi_\vphi)[h]=0$ since $\mathcal{G}(\xi)= \bar{\mathcal{A}}^{-1}(\mathcal{R}(\xi))$. Therefore $\bar{\mathcal{A}}z =[\mathcal{A}z] = [\mathcal{A}h]
    =[(\mathcal{A}-\na \mathcal{R}(\xi_\vphi))[h]]$. 

    Using the Schauder estimate \eqref{eq7}, we get
    \begin{equation}
        ||z||_{C^{2,\alpha}(M)} \leq C ||\bar{\mathcal{A}}z||_Z  \leq ||(\mathcal{A}-\na \mathcal{R}(\xi_\vphi))[h]|| _Y.
    \end{equation}
    Since  $\na \mathcal{R}(0) =\mathcal{A}$, from \eqref{prop1} there exists a modulus of continuity $\omega(r)$ such that
    \begin{equation}
        ||(\mathcal{A}-\na \mathcal{R}(\xi))(h)||_Y\leq w(||\xi||_{C^{2,\alpha}})||h||_{C^{2,\alpha}}.
    \end{equation}
    As a result 
    \begin{equation}
        ||z||_{C^{2,\alpha}(M)} \leq C w(||\xi||_{C^{2,\alpha}}) (||\eta||_{C^{2,\alpha}(M)} +||z||_{C^{2,\alpha}(M)}).
    \end{equation}
    After shrinking $U$, the coefficient of the last term is at most $\frac{1}{2}$, and we obtain
    \begin{equation}
        ||z||_{C^{2,\alpha}(M)} \leq C ||\eta||_{C^{2,\alpha}(M)} .
    \end{equation}
    Recall that all norms are equivalent on $K$. This concludes the proof.
\end{proof}

Here is a notion related to Lyapunov-Schmidt Reduction.
\begin{definition}[Integrability]\label{def: int} {\rm
	A function $v \in \mathcal{CSC}_1$ is said to be {\it integrable} if for all $\vphi \in \ker \nabla^2_\B Q(v)$ there exists a one-parameter family of functions $(v_t)_{t\in (-\delta,\delta)}$, with $v_0 = v$, $\frac{\partial}{\partial t}\big|_{t=0}v_t = \vphi $ and $v_t \in \mathcal{CSC}_1$ for all $t$ sufficiently small.
	}
\end{definition}

Arguing exactly as in \cite{ENL22}, we find that if $v$ is integrable, then $q$ as
given in the above lemma \ref{lem: LS reduction} is constant in a neighborhood of
$0\in K$, and when $v$ is an integrable minimizer, it holds that 
\begin{equation}
    \M_1 \cap \B(v,\delta) = \mathcal{F}:=\{\Psi(\vphi +F(\vphi)) : \vphi \in U\}.
\end{equation}

\section{Local Quantitative Stability of Minimizers}\label{s:localstability}
 In this section, we establish the local version of Theorem \ref{thm: minimizers}.
 For this we need a localized measure of how far $u$ is close to the minimizer $v$.

Given $\delta >0$ and $v \in \mathcal{M}_1$, %$u\in \B$
we let 
\[
d_\delta(u , \mathcal{M}_1) = \frac{\inf\left\{ \| u -\tilde{v}\|_{W^{1,2}(M)} \mid \tilde{v} \in \mathcal{M}_1\cap \B(v,\delta)\right\}}{\|u\|_{W^{1,2}(M)}}.
\]

\begin{proposition}[Local Stability Estimate]\label{lem: Fuglede}
	Let $(M,g)$ be a compact  Riemannian manifold with boundary, and let $v \in \M_1$. Then there exist constants $c, \g$ and $\delta$ depending on $v$ such that 
	\begin{equation}\label{e: Fuglede}
		Q(u) - Y(M,\partial M, [g]) \geq c \,d_\delta(u, \mathcal{M}_1)^{2+\g}\qquad  \text{for all } u \in \B(v,\delta).
			\end{equation}
	If  $v$ is integrable or non-degenerate, then we may take $\g=0$.  
\end{proposition}

\begin{proof}[Proof of Proposition \ref{lem: Fuglede}]
Given $v \in \M_1$, let $\Psi$ be the normalized chart centered at $v$. For $u \in \B(v,\delta)$ sufficiently close to $v$ in $W^{1,2}$, there is a unique small $\xi \in T_v \B$ such that $u= \Psi (\xi)$.

Write $\xi = \vphi + w$, where $\vphi \in K $ and $w\in K^\perp$.
Let F be the Lyapunov–Schmidt map from lemma \ref{lem: LS reduction}, and set $z=w-F(\vphi) \in K^\perp$. We can write 
\begin{equation}
    Q(u)-Y(M,\partial M ,[g])= \hat{Q}(\vphi+w) - \hat{Q}(0) = \big(\hat{Q}(\vphi+w)- \hat{Q}(\vphi+F(\vphi)) \big) +\big(\hat{Q}(\vphi+F(\vphi)) -\hat{Q}(0) \big) = I +II.
\end{equation}
For the first term,
the construction of $F$ gives $\na \hat{Q} (\vphi+F(\vphi))[z]=0$. We use Taylor expand
and see that 
\begin{align}\label{e:finalcoerciveestimate}
\hat{Q}(\vphi+w)-\hat{Q}(\vphi +F(\vphi)) &= \na \hat{Q} (\vphi+F(\vphi))[z] + \frac{1}{2} \na ^2\hat{Q}(\xi)[z,z] \notag\\
&=\frac{1}{2} \na ^2\hat{Q}(0)[z,z] + o(1)||z||^2_{W^{1,2}}\notag \\
&\geq \frac{1}{4} \lambda_1||z||^2_{W^{1,2}},
\end{align}
where we used  the continuity \eqref{prop2} of $\na^2  \hat{Q}(-)$ and $o(1)$ is a quantity that goes to zero as $||u-v||_{W^{1,2}}$ goes to zero, $\lambda_1$ is the smallest non-zero eigenvalue of $\na^2 \hat{Q}(0)$, since its eigenvalue is discrete.

Now we turn to the second term $II$ specializing the discussion depending on $v$ being nondegenerate, integrable or non-integrable.

\smallskip

\noindent {\bf $v$ is non-degenerate.} This is the easiest case, since $\vphi=0$ and then $II = 0$.

\smallskip

\noindent {\bf $v$ is integrable.} By the discussion under the definition \ref{def: int}, since $u\in \B(v,\delta)$ is form of $\Psi(\vphi+F(\vphi))$,  we have $II=0$.

\smallskip

\noindent {\bf $v$ is non-integrable.} Recall that $q(\vphi)= \hat{Q}(\vphi+F(\vphi))$. We know that $\vphi \mapsto q(\vphi)$ is an analytic function  $\mathbb R^\ell\rightarrow \mathbb R$ where $\ell = \dim K$. Thus we can apply the \L ojasiewicz inequality\cite{Lio65} : 

\begin{lemma}\label{lo}
Let $q: \mathbb R^\ell \rightarrow \mathbb R$ be a real analytic function and assume that $\nabla q(0) = 0$. Then there exist  $\tilde{\delta} > 0, c > 0$ and $\g > 0$ (all of which depend on $q$ and on the critical point $0$)
%\Luca{We mean on $0$?} ) 
such that for all $\vphi \in B(0,\tilde{\delta})$,
 \begin{equation}\label{e: lojdistance}
|q(\vphi) - q(0)| \geq c \inf\left\{ |\vphi - \bar \vphi | \ : \ \bar \vphi \in B(0,\tilde{\delta}), \, \na q(\bar \vphi ) =0\right\}^{2+\g}.
\end{equation}
\end{lemma}

 Appealing to the definition of $q$ in Lemma~\ref{lem: LS reduction} and   the \L ojasiewicz inequality in  Lemma~\ref{lo}, we see that 
\begin{equation}\label{eqn: intermediate Loj} 
\begin{split}
	II& =
	q(\vphi) - q(0) \\
	& \geq c \inf\{ |\vphi - \hat{ \vphi} | \ : \  \hat{\vphi} \in K\cap B(0,\delta), \na q(\hat\vphi ) =0\}^{2+\g}.
\end{split}	
\end{equation}
It remains to compare the term $|\vphi - \hat{ \vphi} |$ with the distance from $u$ to the local minimizing set. For any $u'\in \M_1 \cap B_\delta(v)$, from the property \eqref{item: prop2} of Lemma \ref{lem: LS reduction} we can write $u'= \Psi(\vphi'+F(\vphi'))$. Hence
\begin{align*}
    ||u-u'||_{W^{1,2}} &\leq C ||\vphi + F(\vphi) + z  - \vphi'-F(\vphi')||_{W^{1,2}}\\
    &\leq C(||\vphi - \vphi'||_{W^{1,2}} +||z||_{W^{1,2}} +||F(\vphi)- F(\vphi')||_{W^{1,2}})\\
    &\leq C(||\vphi - \vphi'||_{W^{1,2}} +||z||_{W^{1,2}}).
\end{align*}
Here we used the property \eqref{eqn: estimates for DF}.
Together with \eqref{e:finalcoerciveestimate}, we have finished the proof.

\end{proof}

\section{Proofs of Theorems~\ref{thm: minimizers}}

In this section, we will prove our main results, that is, Theorems \ref{thm: minimizers}. In section~\ref{s:localstability}, we have proved the local version. As a result, Theorem \ref{thm: minimizers} comes from the compactness argument.(See, e.g. \cite{Ar04} \cite{HL99})

\begin{lemma}\label{lem: cpt}
	Let $(M,g)$ be a compact Riemannian manifold of dimension $n\geq 3$ with boundary such that $Y(M,\partial M,[g]) < Y(S^n_+,\partial S^n_+,[g_0]) $ and let $(u_i)\subset \B$ be a sequence such that $Q(u_i) \to Y$.  Then, up to a subsequence, $u_i$ converges strongly in $W^{1,2}(M)$ to some $v \in \mathcal{M}_1$. 
\end{lemma}
We now prove Theorem~\ref{thm: minimizers}.
\begin{proof}[Proof of Theorem~\ref{thm: minimizers}]
Since both sides of \eqref{e: minimizers} are zero-homogeneous in $u$ and because  $\inf\{ \| u -v\|_{W^{1,2}(M)} : v\in \M_1\} \geq d(u, \M)$,  we may work in $\B$ without loss of generality.

Given $v \in \M_1,$ let $\delta(v)$,  $\g(v)$, and $c(v)$ be the constants given in Proposition~\ref{lem: Fuglede}.
%and recall fromthat these constants are invariant under multiplying $v$ by a constant. Therefore, 
Since the set $\M_1 = \{  v \in \M : \| v\|_{L^{2^*}(M)}=1\}$ is compact in $W^{1,2}$ by Lemma~\ref{lem: cpt}, we may cover $\M_1$ by balls $\B(v, \delta(v)/2)$ and take a finite subcover $\{\B(v_i, \delta(v_i)/2)\}_{i\in \mathscr I}$. Then we define
\begin{align*}
\delta_0 &= \min_{i\in \mathscr I} \delta(v_i)/2 > 0,\\
\g_0 & = \max_{i\in \mathscr I}  \g(v_i)< \infty,\\
c_0 & = \min_{i\in \mathscr I}  c(v_i) >0.
\end{align*}

Let $u\in \B$ ,We now consider two cases:  either $d(u, \mathcal{M}_1) \leq \delta_0/4$ or $d(u, \mathcal{M}_1) > \delta_0/4$.
If $d(u, \mathcal{M}_1) \leq \delta_0/4$,  
there exists an $i \in \mathscr I$ such that $\|u -v_i\|_{W^{1,2}} < \delta_i/2$. We may take $\tilde{v}$ as the closest element of $\M_1$ to $u$ and use the triangle inequality to get  $\|\tilde{v}-v_i\|_{W^{1,2}} < \delta_i$. Thus local quantitative stability implies, cf. proposition \ref{lem: Fuglede}
$$Q(u) - Y(M,,\partial M,[g]) \geq c(v_i)d_{\delta_i}(u, \mathcal{M}_1)^{2+\g_i} \geq c_0d(u, \mathcal M_1)^{2+ \g_0}.
$$ 

On the other hand, we consider the case $d(u, \mathcal M_1) > \delta_0/4$. By (the contrapositive of) Lemma~\ref{lem: cpt}, we get $\exists \e>0$ depends on $\delta_0$ such that
\begin{equation}
    d(u, \mathcal M_1) > \delta_0/4 \Rightarrow Q(u)-Y(M,[g]) > \varepsilon.
\end{equation}
Moreover, observing that by definition, $d(u, \mathcal M) \leq  1$, we have $Q(u)-Y(M,\partial M,[g]) \ge \epsilon d(u, \M)^{2+\gamma_0}$.

 Letting $c = \min\left\{c_0, \varepsilon\right\}$ we have proven the stability estimate \eqref{e: minimizers} for all $u \in \B$. 

\end{proof}
\textbf{Acknowledgments}
 \quad We would like to thank Prof.Yuxin Ge for his advanced guidance.

\end{document}